\documentclass{amsart}
\usepackage[utf8x]{inputenc}
\usepackage{amsmath, amsthm, amssymb}
\usepackage[usenames,dvipsnames,svgnames,table]{xcolor}
\usepackage[margin=1.3in]{geometry}
\usepackage[french,english]{babel}

\newtheorem{theorem}{Theorem}[section]
\newtheorem{proposition}[theorem]{Proposition}
\newtheorem{lemma}[theorem]{Lemma}

\newcommand{\R}{\mathbb R}

\newcommand{\sgn}{\mbox{sgn}}

\newcommand{\osc}{\mbox{osc}}

\numberwithin{equation}{section}

\newcommand{\commentout}[1]{}

\title[Global well-posedness for the 2D Muskat problem]{Global well-posedness for the 2D Muskat problem with slope less than 1.}
\author{Stephen Cameron}
\address{Department of Mathematics, University of Chicago, 5734 S. University Ave., Chicago, IL 60637}
\email{scameron@math.uchicago.edu}

\begin{document}

%

\begin{abstract}
We prove the existence of global, smooth solutions to the 2D Muskat problem in the stable regime whenever the product of the maximal and minimal slope is less than 1.  The curvature of these solutions solutions decays to 0 as $t$ goes to infinity, and they are unique when the initial data is $C^{1,\epsilon}$.  We do this by getting a priori estimates using a nonlinear maximum principle first introduced in \cite{Kiselev}, where the authors proved global well-posedness for the surface quasi-geostraphic equation.
\end{abstract}

\maketitle

\section{Introduction}

The Muskat problem was originally introduced by Muskat in \cite{Muskat} in order to model the interface between water and oil in tar sands.  In general, it describes the interface between two incompressible, immiscible fluids of different constant densities in a porous media.  The fluids evolve according to Darcy's law, giving an evolution of the interface (see \cite{Localwellpose} for derivation of equations), and in 2D is analogous to the two phase Hele-Shaw cell (see \cite{Hele}).  In the case that the two fluids are of equal viscosity and the interface is given by the graph $y=f(t,x)$ with the denser fluid on bottom (i.e. the stable regime), the function $f$ satisfies
\begin{equation}
f_t(t,x) = \int\limits_{\R} \frac{(f_x(t,y) - f_x(t,x))(y-x)}{(f(t,y)-f(t,x))^2 + (y-x)^2} dy,
\end{equation}
after the appropriate renormalization.  By making a change of variables, (see the proof of Lemma 5.1 of \cite{MaxPrinciple}) we get the equivalent system
\begin{equation}\label{e:fequation}
f_t(t,x) = \int\limits_{\R} \frac{f(t,y)-f(t,x) - (y-x)f_x(t,x)}{(f(t,y)-f(t,x))^2 + (y-x)^2} dy,
\end{equation}
which will be more useful for our purposes.  Since the function $f$ is Lipschitz, the above integral can be viewed as a nonlinear perturbation of the half Laplacian.  In fact, it is easy to see that linearizing around a flat solution gives
\begin{equation}
f_t(t,x) = -c(-\Delta)^{1/2}f(t,x),
\end{equation}
demonstrating the natural parabolicity of the problem.


The Muskat problem is known to be locally well-posed in $H^k$ for $k\geq 3$ with solutions satisfying $L^\infty$ and $L^2$ maximum principles, but neither imply any gain of derivatives (see \cite{MaxPrinciple}, \cite{Globalold}).


Under the assumption $||f_0'||_{L^\infty}<1$, there have been a number of positive results.  In \cite{Globalold} the authors prove an $L^\infty$ maximal principle for the slope $f_x$ along with the existence of global weak Lipschitz solutions using a regularized system.   Recently, \cite{Energy} improved the $L^2$ energy estimate of \cite{Globalold} (which holds for any solution) to one analogous with the energy estimate from the linear equation under this assumption on the slope.  When the initial data $f_0\in H^2(\R)$ with $||f_0||_1 = ||\ |\xi| \hat{f}_0(\xi)||_{L_\xi^1} $ less than some explicit constant $\approx 1/3$ (which implies slope less than 1), \cite{Global} proves that a unique global strong solution exists.  In this case \cite{Decay} proves optimal decay estimates on the norms $||f(t,\cdot)||_s = ||\ |\xi|^s \hat{f}(t,\xi)||_{L_\xi^1}$, matching the estimates for the linear equation.

Recently, \cite{Monotonic} was also able to prove the existence of global weak solutions for arbitrarily large monotonic initial data.  They did this using the regularized system from \cite{Globalold} to prove that both $f$ and $f_x$ still obey the maximum principle under this monotonicity assumption.

Because solutions to \eqref{e:fequation} have the natural scaling $\displaystyle\frac{1}{r}f(rt,rx)$, we see that $L^\infty$ or sign bounds on the slope $f_x$ are scale invariant properties.  We fit these two types of assumptions into the same framework by showing that the critical quantity is in fact the product of the maximal and minimal slopes,
\begin{equation}
\beta(f_0'):= (\sup\limits_x f_0'(x))(\sup\limits_y -f_0'(y)).
\end{equation}
As we shall see in section 3, the derivative $f_x$ obeys the equation
\begin{equation}\label{e:gequationK2}
(f_x)_t(t,x) = f_{xx}(t,x) \int\limits_{\R}\frac{-h}{\delta_hf(t,x)^2+h^2} dh + \int\limits_\R \delta_hf_x(t,x) K(t,x,h) dh.
\end{equation}
where $\delta_hf(t,x):=f(t,x+h)-f(t,x)$ and the kernel $K$ is uniformly elliptic of order 1 whenever $\beta(f_0')<1$.  Thus we naturally get regularizing effects from the equation whenever the initial data satisfies this bound.  It's clear that $||f_0'||_{L^\infty}<1$ implies $\beta(f_0')<1$, and for bounded monotonic data we get that $\beta(f_0')=0$ since either $\sup f_0'=0$ or $\inf f_0' =0$.  Thus this $\beta(f_0')<1$ provides a natural interpolation between these two types of assumptions.


In contrast to the positive results, \cite{Breakdown} shows that there is an open subset of initial data in $H^4$ such that the Rayleigh-Taylor condition breaks down in finite time.  That is, $\lim\limits_{t\to t_0-} ||f_x(t,\cdot)||_{L^\infty} = \infty$ for some time $t_0$, after which the interface between the fluids can no longer be described by a graph.

The authors of \cite{ConstantinMain} made great progress towards proving global regularity.  They proved that if the initial data $f_0\in H^k$, then the solution $f$ will exist and remain in $H^k$ so long as the slope $f_x(t,\cdot)$ remains bounded and uniformly continuous.  Thus the natural next step is to prove the generation of a modulus of continuity for $f_x$, hence

\begin{theorem}\label{t:main}
Let $f_0\in W^{1,\infty}(\R)$ with
\begin{equation}
\beta(f_0'):=(\sup\limits_x f_0'(x))(\sup\limits_y -f_0'(y))<1.
\end{equation}
Then there exists a classical solution
\begin{equation}
f\in C([0,\infty)\times \R)\cap C^{1,\alpha}_{loc}((0\,\infty)\times \R)\cap L^\infty_{loc}((0,\infty);C^{1,1}),
\end{equation}
to \eqref{e:fequation} with $f_x$ satisfying both the maximum principle and
\begin{equation}\label{e:mainresult}
f_x(t,x)-f_x(t,y) \leq \rho\left(\frac{|x-y|}{t}\right), \quad t>0, x\not = y \in \R,
\end{equation}
for some Lipschitz modulus of continuity $\rho$ depending solely on $\beta(f_0')$,$||f_0'||_{L^\infty}$.
In the case that $f_0\in C^{1,\epsilon}(\R)$ for some $\epsilon>0$, then the solution $f$ is unique with $f\in L^\infty([0,\infty);C^{1,\epsilon})$.
\end{theorem}
The uniqueness statement follows essentially from the uniqueness theorem of \cite{ConstantinMain}.  We note in the appendix the few small changes needed to their proof in order to apply it here.

The most vital part of Theorem \ref{t:main} is the spontaneous generation of the modulus $\rho(\cdot/t)$, as everything else will follow from that.  The spontaneous generation/propogation of a general modulus of continuity has old roots as classical Holder estimates, but its only recently that the idea to tailor make moduli for specific equations emerged.  The technique first appeared in \cite{Kiselev}, where the authors used it to prove global well-posedness for the surface quasi-geostraphic equation.  It has had great success at proving regularity for a number of active scalar equations, that is equations of the form
\begin{equation}
\theta_t + (u\cdot \nabla)\theta + \mathcal{L}\theta=0,
\end{equation}
where $u$ is a flow depending on $\theta$ and $\mathcal{L}$ is some diffusive operator.  See \cite{KiselevSurvey}, \cite{Supercrit} for a good overview of results using this method.

To date, these tailor made moduli have only been applied to cases where all the nonlinearity has been in the flow velocity $u$, and the diffusive term $\mathcal{L}$ has been rather nice (typically $(-\Delta)^{\alpha}$, or at least a Fourier multiplier). We will be applying this method to $f_x$, which solves the active scalar equation \eqref{e:gequationK2}.  Note that in this equation, the kernel $K$ defined in \eqref{e:Kequation} is a highly nonlinear function of $f, f_x$.  Thus this is the first time the method has been applied in a fully nonlinear equation.


%
%
%
%

We prove Theorem \ref{t:main} by deriving a priori estimates for smooth solutions to \eqref{e:fequation} with initial data $f_0\in C^\infty_c(\R)$ depending primarily on $\beta(f_0'),||f_0'||_{L^\infty}$.  We prove enough estimates that by approximating in $W^{1,\infty}_{loc}$ with smooth compactly supported initial data, we get solutions $f^\epsilon$ which will converge along subsequences in $C^1_{loc}$ to a solution $f$ solving \eqref{e:fequation} for arbitrary initial data $f_0\in W^{1,\infty}(\R)$ with $\beta(f_0')<1$.

The rest of the paper is organized as follows.  We begin by repeating the breakthrough argument of \cite{Kiselev} in Section 2.  In Section 3, we differentiate \eqref{e:fequation} to derive the equation for $f_x$, showing that it satisfies the maximum principle when $\beta(f_0')<1$.  In Section 4, we state how a modulus of continuity $\omega$ interacts with the equation in our main technical lemma.  In Sections 5 and 6 we then derive the bounds on the drift and diffusion terms necessary to prove that lemma.  In Section 7, we apply our main technical lemma to a specific modulus of continuity, and finally in Section 8 we complete the proof of \eqref{e:mainresult} by choosing the correct modulus $\rho$.  In Section 9, we then use \eqref{e:mainresult} to prove a few estimates on regularity in time, guaranteeing enough compactness to prove that there are classical solutions for rough initial data.  Finally in the appendix, we give a quick outline for how to modify the uniqueness proof of \cite{ConstantinMain} to work for initial data $f_0\in C^{1,\epsilon}(\R)$ with $\beta(f_0')<1$.

\section{Breakthrough Scenario}

Assume that $f_0\in C^\infty_c(\R)$ with $\beta(f_0')<1$, so that there exists a solution $f\in C^1((0,T_+); H^k)$ for $k$ arbitrarily large and some $T_+>0$ by \cite{MaxPrinciple}.  Note that under the assumption that $\beta(f_0')<1$, we will show that the maximum principle holds (see Section 3 Proposition \ref{p:maximumprinciple}) and hence $||f_x||_{L^\infty([0,T_+)\times \R)}\leq ||f_0'||_{L^\infty}$ is uniformly bounded.  Fix a Lipschitz modulus $\rho$ which we will define later. For sufficiently small times, $f_x(t,\cdot)$ will have modulus $\rho(\cdot/t)$ since it is smooth and bounded.  It then follows by the main theorem of \cite{ConstantinMain} that as long as $f_x(t,\cdot)$ continues to have modulus $\rho(\cdot/t)$, the solution $f$ will exist with $T_+>t$.

So, we proceed as in \cite{Kiselev}'s proof for quasi-geostraphic equation.  Suppose that $f_x(t,\cdot)$ satisfies \eqref{e:mainresult} for all $t<T$.  Then by continuity,
\begin{equation}
f_x(T,x)-f_x(T,y) \leq \rho\left(\frac{|x-y|}{T}\right), \quad \forall x\not = y\in \R.
\end{equation}
We first prove that if we have the strict inequality $f_x(T,x)-f_x(T,y) < \rho\left(|x-y|/T\right)$, then $f_x(t,\cdot)$ will have modulus $\rho(\cdot/ t)$ for $t\leq T+\epsilon$.

\begin{lemma}
Let $f\in C([0,T_+); C^3_0(\R))$, and $T\in (0,T_+)$.  Suppose that $f(T,\cdot)$ satisfies
\begin{equation}
f_x(T,x)-f_x(T,y) < \rho\left(|x-y|/T\right), \forall x\not = y\in \R,
\end{equation}
for some Lipschitz modulus of continuity $\rho$ with $\rho''(0) = -\infty$.  Then
\begin{equation}
f_x(T+\epsilon, x) - f_x(T+\epsilon, y) < \rho(|x-y| / (T+\epsilon)), \forall x\not = y\in \R,
\end{equation}
for all $\epsilon>0$ sufficiently small.
\end{lemma}

\begin{proof}
To begin, note that for any compact compact subset $K\subset \R^2\setminus \{(x,x)|x\in \R\}$,
\begin{equation}
f_x(T,x)-f_x(T,y) < \rho(|x-y|/T) \quad \forall (x,y)\in K \quad \Rightarrow \quad f_x(T+\epsilon, x)-f_x(T+\epsilon,y) < \rho(|x-y|/(T+\epsilon)) \quad \forall (x,y)\in K,
\end{equation}
for $\epsilon >0$ sufficiently small by uniform continuity.  So, we only need to focus on pairs $(x,y)$ that are either close to the diagonal, or that are large.

To handle $(x,y)$ near the diagonal, we start by noting that $f(T,\cdot)\in C^3(\R)$ and $\rho''(0)=-\infty$.  Thus for every $x$ we get that
\begin{equation}
|f_{xx}(T,x)| < \frac{\rho'(0)}{T}.
\end{equation}
Since $f\in C([0,T_+); C^3_0(\R))$, $f_{xx}(T,x)\to 0$ as $x\to \infty$.  Thus we can take the point where $\max\limits_x |f_{xx}(T,x)|$ is achieved to get that
\begin{equation}
||f_{xx}(T,\cdot)||_{L^\infty} < \frac{\rho'(0)}{T}.
\end{equation}
By continuity of $f_{xx}$, we thus have $||f_{xx}(T+\epsilon,\cdot)||_{L^\infty} < \displaystyle\frac{\rho'(0)}{T+\epsilon}$ for $\epsilon>0$ sufficiently small.  Hence,
\begin{equation}\label{e:smallxypropogation}
f_x(T+\epsilon,x)-f_x(T+\epsilon,y) < \rho\left(\frac{|x-y|}{T+\epsilon}\right), \quad |x-y|<\delta,
\end{equation}
for $\epsilon, \delta$ sufficiently small.

Now let $R_1,R_2>0$ be such that
\begin{equation}
\rho(R_1/ (T+\epsilon)) > \osc_\R f_x(T+\epsilon,\cdot),
\end{equation}
and that $|x|>R_2$ implies
\begin{equation}
|f_x(T+\epsilon,x)| < \frac{\rho(\delta/(T+\epsilon))}{2},
\end{equation}
for $\epsilon>0$ sufficiently small.
Taking $R = R_1+R_2$, it's easy to check that $|x|>R$ implies that
\begin{equation}
|f_x(T+\epsilon,x) - f_x(T+\epsilon,y)| < \rho(|x-y|/(T+\epsilon)), \quad \forall y\not = x.
\end{equation}

Finally, taking $K= \{ (x,y)\in \R^2: |x-y|\geq \delta, x,y\in \overline{B_R}\}$, we're done.

\end{proof}

Thus by the lemma, if $f_x$ was to lose its modulus after time $T$, we must have that there exist $x\not = y\in \R$ with
\begin{equation}\label{e:modulusequality}
f_x(T,x) - f_x(T,y) = \rho\left(\frac{|x-y|}{T}\right).
\end{equation}
We will show for a smooth solution $f$ of \eqref{e:fequation} and the correct choice of $\rho$ that in this case
\begin{equation}\label{e:finalinequality}
\frac{d}{dt} \left(f_x(t,x) - f_x(t,y)\right)\bigg|_{t=T} < \frac{d}{dt}\left(\rho\left(\frac{|x-y|}{t}\right)\right)\bigg|_{t=T},
\end{equation}
contradicting the fact that $f_x$ had modulus $\rho(\cdot /t)$ for time $t<T$.

Thus we just need to prove \eqref{e:finalinequality} to complete the proof of the generation of modulus of continuity \eqref{e:mainresult} of Theorem \ref{t:main}.

\section{Equation for $f_x$}

So, we just need to prove \eqref{e:finalinequality}.  To begin, we need to examine the equation that $f_x$ solves.
Since everything we will be doing is for some fixed time $T>0$, we will suppress the time variable from now on.  Differentiating  \eqref{e:fequation}, we see that $f_x$ solves
\begin{equation}\label{e:gequationy}
\begin{split}
(f_x)_t(x) = f_{xx}(x) &\int\limits_{\R} \frac{x-y}{(f(y)-f(x))^2 + (y-x)^2} dy
\\&+ \int\limits_{\R}\left(f(y)-f(x) - (y-x)f_x(x)\right)\frac{2\left((f(y)-f(x))f_x(x)+(y-x)\right)}{\left((f(y)-f(x))^2 + (y-x)^2\right)^2} dy.
\end{split}
\end{equation}

To simplify notation, we reparametrize \eqref{e:gequationy} by taking $y=x+h$, and letting $$\delta_hf(x):= f(x+h) -f(x),$$ we get
\begin{equation}\label{e:gequationh}
\begin{split}
(f_{x})_t(x) = f_{xx}(x) &\int\limits_{\R} \frac{-h}{(\delta_hf(x))^2 + h^2} dh
\\&+ \int\limits_{\R}\left(\delta_hf(x) - hf_x(x)\right)\frac{2\left(\delta_hf(x)f_x(x)+h\right)}{\left(\delta_hf(x)^2 + h^2\right)^2} dh.
\end{split}
\end{equation}

Note that $$\delta_hf(x) - hf_x(x) = \int\limits_0^h \delta_s f_x(x) ds,$$ for $h>0$, and $$\delta_hf(x) - hf_x(x) = -\int\limits_h^0 \delta_s f_x(x) ds,$$ for $h<0$.

With that in mind, define
\begin{equation}\label{e:kequation}
k(x,s) = \frac{2\left(\delta_sf(x)f_x(x)+s\right)}{\left(\delta_sf(x)^2 + s^2\right)^2},
\end{equation}
and
\begin{equation}\label{e:Kequation}
K(x,h) = \left\{ \begin{array}{cl} \int\limits_h^\infty k(x,s)ds, & h>0 \\ \int\limits_{-\infty}^h -k(x,s) ds, & h<0 \end{array}\right. .
\end{equation}
Then integrating \eqref{e:gequationh} by parts, we have that $f_x$ solves the equation
\begin{equation}\label{e:gequationK}
(f_x)_t(x) = f_{xx}(x) \int\limits_{\R}\frac{-h}{\delta_hf(x)^2+h^2} dh + \int\limits_\R \delta_hf_x(x) K(x,h) dh.
\end{equation}

As
\begin{equation}
\frac{-\beta(f_x)}{s}\leq \frac{f_x(x)\delta_sf(x)}{s}\leq  \frac{||f_x||_{L^\infty}^2}{s},
\end{equation}
we see that
$$\frac{2(1-\beta(f_x))}{(1+||f_x||_{L^\infty}^2)^2} \frac{1}{|s|^3} \leq \sgn(s) k(x,s)\leq \frac{2(1+||f_x||_{L^\infty}^2)}{|s|^3},$$
and hence
\begin{equation}
\frac{1-\beta(f_x)}{(1+||f_x||_{L^\infty}^2)^2} \frac{1}{h^2} \leq K(x,h)\leq \frac{1+||f_x||_{L^\infty}^2}{h^2}.
\end{equation}
Thus in the case that $\beta(f_x)\leq 1$, we then have that the kernel $K$ is a nonnegative, from which we get immediately
\begin{proposition}\label{p:maximumprinciple}(Maximum Principle)

Let $f_x$ be a sufficiently smooth solution to \eqref{e:gequationK} with $\beta(f_0')\leq 1$. Then for any $0\leq s\leq t$, we have that
\begin{equation}
\inf\limits_y f_x(s,y)\leq \inf\limits_y f_x(t,y)\leq \sup\limits_y f_x(t,y)\leq \sup\limits_y f_x(s,y).
\end{equation}
\end{proposition}

In particular, since $\beta(f_0')<1$ the maximum principle tells us that
\begin{equation}
\beta(f_x)\leq \beta(f_0')<1, \qquad \qquad ||f_x||_{L^\infty}\leq ||f_0'||_{L^\infty}<\infty.
\end{equation}

Thus we get that
\begin{equation}\label{e:Klowerbound}
0<\frac{\lambda}{h^2}\leq K(x,h) \leq \frac{\Lambda}{h^2},
\end{equation}
where
\begin{equation}\label{e:lambda}
\lambda = \frac{1-\beta(f_0')}{(1+||f_0'||_{L^\infty}^2)^2}, \quad \Lambda = 1+||f_0'||_{L^\infty}^2.
\end{equation}
Thus $K$ is comparable to the kernel for $(-\Delta)^{1/2}$, so $f_x$ solves the uniformly elliptic equation \eqref{e:gequationK}.  Note that the sole reason we require $\beta(f_0')<1$ is to ensure this ellipticity of $K$.


\section{Moduli Estimates}

Our goal is to show that if $f_x(T,\cdot)$ has modulus $\rho(\cdot/T)$ and equality is achieved at two points \eqref{e:modulusequality}, then \eqref{e:finalinequality} must hold, contradicting the assumptions of the breakthrough argument (see section 2).  To that end, we first need to understand how a modulus of continuity interacts with the equation for $f_x$ \eqref{e:gequationK}.  Hence,
\begin{lemma}\label{l:omegabound}
Let $f: [0,\infty)\times \R\to \R$ be a bounded smooth solution to \eqref{e:fequation} with $\beta(f_0')<1$, and $\omega: [0,\infty)\to [0,\infty)$ be some fixed modulus of continuity.  Assume that at some fixed time $T$ that
\begin{equation}\label{e:omegaassumptions}
\begin{split}
\delta_hf_x(T,x)\leq \omega(|h|),
\\ f_x(T,\xi/2)-f_x(T,-\xi/2) = \omega(\xi),
\end{split}
\end{equation}
for all $h\in \R$, and for some $\xi>0$.  Then
\begin{equation}\label{e:omegabound}
\begin{split}
\frac{d}{dt}(f_x(t,\xi/2)-f_x(t,-\xi/2))\bigg|_{t=T} \leq & A\omega'(\xi)\left(\int\limits_0^\xi \frac{\omega(h)}{h}dh + \xi\int\limits_{\xi}^\infty \frac{\omega(h)}{h^2} dh  + \ln(M+1)\omega(\xi)\right)
\\&+A\omega(\xi)\int\limits_{M\xi}^\infty \frac{\omega(h)}{h^2}dh +2(\Lambda-\lambda)\int\limits_{\xi}^{M\xi} \frac{(\omega(h-\xi) - \omega(\xi))_+}{h^2}dh
\\&+2\lambda\int\limits_0^\xi \frac{\delta_h\omega(\xi) + \delta_{-h}\omega(\xi)}{h^2}dh + 2\lambda\int\limits_{\xi}^\infty \frac{\omega(h+\xi)-\omega(h)-\omega(\xi)}{h^2}dh,
\end{split}
\end{equation}
for any $M\geq 1$, where $A$ depends only on $||f_0'||_{L^\infty}$ and $\lambda, \Lambda$ are as in \eqref{e:lambda}.
\end{lemma}

This is the main technical lemma that we need.  Since solutions to \eqref{e:fequation} are closed under translation and sign change, it suffices to consider the above situation for our proof of \eqref{e:finalinequality}.

Note that (4.2) holds for any value of the parameter $M\geq 1$.  Later in Lemma 6.1, we will essentially use two different values of $M$ depending on the size of $\xi$.  In the small $\xi$ regime we can simply take $M=1$, but in the large $\xi$ regime we will need to take $M$ to be a sufficiently large constant depending only on initial data (but not on exact size of $\xi$) in order to control the size of the error term $\omega(\xi)\int\limits_{M\xi}^\infty \frac{\omega(h)}{h^2}dh$.

The proof for Lemma \ref{l:omegabound} is essentially a nondivergence form argument; our function $f_x$ is touched from above at $\xi/2$ by our modulus $\omega$, and its touched from below at $-\xi/2$ by $-\omega$.  Specifically,
\begin{equation}\label{e:touchfromabove}
\begin{split}
\delta_h f_x(\xi/2)\leq \delta_h\omega(\xi), \qquad \forall h>-\xi,
\\ \delta_h f_x(-\xi/2) \geq -\delta_{-h}\omega(\xi), \qquad \forall h<\xi.
\end{split}
\end{equation}
From \eqref{e:touchfromabove}, we want to derive as much information as we can and bound $\displaystyle\frac{d}{dt}(f_x(\xi/2)-f_x(-\xi/2))$.  To that end, by dividing \eqref{e:touchfromabove} through by $h$ and taking the limit as $h\to 0$, we then get that
\begin{equation}
f_{xx}(\xi/2) = f_{xx}(-\xi/2) = \omega'(\xi).
\end{equation}

Hence by our equation for $f_x$ \eqref{e:gequationK}, we have that
\begin{equation}\label{e:driftanddiffusion}
\begin{split}
\frac{d}{dt}(f_x(\xi/2)&-f_x(-\xi/2)) = \omega'(\xi)\int\limits_\R \left(\frac{-h}{\delta_hf(\xi/2)^2+h^2} - \frac{-h}{\delta_{h}f(-\xi/2)^2+h^2}\right)dh
\\&\qquad\qquad\qquad\qquad\qquad + \int\limits_{\R} \delta_hf_x(\xi/2)K(\xi/2,h) - \delta_hf_x(-\xi/2)K(-\xi/2,h)dh
\\&=\omega'(\xi)\int\limits_\R \left(\frac{-h}{\delta_hf(\xi/2)^2+h^2} - \frac{-h}{\delta_{h}f(-\xi/2)^2+h^2}\right)dh
+\omega'(\xi)\int\limits_{-M\xi}^{M\xi} \left(hK(\xi/2,h) - hK(-\xi/2,h)\right)dh
\\&\quad + \int\limits_{-M\xi}^{M\xi} (\delta_hf_x(\xi/2)-h\omega'(\xi))K(\xi/2,h) - (\delta_hf_x(-\xi/2)-h\omega'(\xi))K(-\xi/2,h)dh
\\&\quad +\int\limits_{|h|>M\xi} \delta_hf_x(\xi/2)K(\xi/2,h) -\delta_hf_x(-\xi/2)K(-\xi/2,h)dh,
\end{split}
\end{equation}
for any $M\geq 1$.  The first two terms of the RHS of \eqref{e:driftanddiffusion} act as a drift, giving rise to the first two error terms of \eqref{e:omegabound}.  The latter two terms of \eqref{e:driftanddiffusion} act as a diffusion, giving rise to both the helpful (negative) terms in \eqref{e:omegabound}, as well as additional error terms (the middle terms of \eqref{e:omegabound}) arising from the difference in the kernels, $|K(\xi/2,h)-K(-\xi/2,h)|$.

\section{Bounds on Drift terms}

We begin proving Lemma \ref{l:omegabound} by bounding the drift terms of \eqref{e:driftanddiffusion}, starting with

\begin{lemma} \label{l:driftone}
Under the assumptions of Lemma \ref{l:omegabound},

\begin{equation} \label{e:firstdrifttermfinal}
\omega'(\xi)\bigg|\int\limits_{\R}\frac{-h}{\delta_hf(\xi/2)^2+h^2} - \frac{-h}{\delta_{h}f(-\xi/2)^2+h^2}dh\bigg| \lesssim \omega'(\xi)\left(\int\limits_0^\xi \frac{\omega(h)}{h}dh + \xi\int\limits_\xi^\infty \frac{\omega(h)}{h^2}dh\right).
\end{equation}
\end{lemma}

\begin{proof}
We want to bound \eqref{e:firstdrifttermfinal} by symmetrizing the kernels for $|h|<\xi$, and and then using the continuity in the first variable for $|h|>\xi$.  To that end,
\begin{equation}\label{e:firstdriftterm}
\begin{split}
\omega'(\xi)\int\limits_\R &\left(\frac{-h}{\delta_hf(\xi/2)^2+h^2} - \frac{-h}{\delta_{h}f(-\xi/2)^2+h^2}\right)dh
\\&\leq \omega'(\xi)\int\limits_{0}^\xi h\bigg|\frac{\delta_hf(\xi/2)^2 - \delta_{-h}f(\xi/2)^2}{(\delta_hf(\xi/2)^2+h^2)(\delta_{-h}f(\xi/2)^2+h^2)} + \frac{\delta_hf(-\xi/2)^2 - \delta_{-h}f(-\xi/2)^2}{(\delta_hf(-\xi/2)^2+h^2)(\delta_{-h}f(-\xi/2)^2+h^2)} \bigg|dh
\\& \quad + \omega'(\xi)\int\limits_{|h|>\xi} |h|\bigg|\frac{\delta_hf(\xi/2)^2 - \delta_hf(-\xi/2)^2}{(\delta_hf(\xi/2)^2+h^2)(\delta_hf(-\xi/2)^2+h^2)} \bigg|dh .
\end{split}
\end{equation}

We bound the first integral using
\begin{equation}\label{e:fbound1}
\begin{split}
|\delta_h f(x)|\lesssim |h|,
\\ |\delta_h f(x)+\delta_{-h}f(x)| = \bigg|\int\limits_{0}^h f_x(x+s)-f_x(x+s-h) ds\bigg|\leq \omega(h)h,
\end{split}
\end{equation}
Thus get that for $0\leq h<\xi$,
\begin{equation}
\bigg|\frac{\delta_hf(x)^2-\delta_{-h}f(x)^2}{(\delta_hf(x)^2+h^2)(\delta_{-h}f(x)^2+h^2)}\bigg| \lesssim \frac{\omega(h)}{h^2},
\end{equation}
and hence
\begin{equation}\label{e:I1bound1}
\int\limits_{0}^\xi h\bigg|\frac{\delta_hf(\xi/2)^2 - \delta_{-h}f(\xi/2)^2}{(\delta_hf(\xi/2)^2+h^2)(\delta_{-h}f(\xi/2)^2+h^2)} dh \bigg|  \lesssim \int\limits_0^\xi \frac{\omega(h)}{h}dh.
\end{equation}
For $|h|\geq \xi$, we bound $|\delta_h f(\xi/2)+\delta_hf(-\xi/2)|\lesssim |h|$ and
\begin{equation}\label{e:fbound2}
\begin{split}
\bigg|\delta_h f(\xi/2) &- \delta_h f(-\xi/2) \bigg|=\bigg| \displaystyle\int\limits_0^h f_x(\xi/2+s)-f_x(-\xi/2+s)ds \bigg|
\\&= \bigg|\int\limits_{0}^{\xi}f_x(h-\xi/2+s) - f_x(-\xi/2+s)ds\bigg|\leq \xi\omega(|h|),
\end{split}
\end{equation}
in order to get
\begin{equation}\label{e:I1bound2}
\int\limits_{|h|>\xi} |h|\bigg|\frac{\delta_hf(\xi/2)^2-\delta_{h}f(-\xi/2)^2}{(\delta_hf(\xi/2)^2+h^2)(\delta_{h}f(-\xi/2)^2+h^2)}\bigg|dh \lesssim \xi\int\limits_{\xi}^\infty \frac{\omega(h)}{h^2}dh.
\end{equation}

Putting \eqref{e:I1bound1} and \eqref{e:I1bound2} together, we thus have
\begin{equation}
\omega'(\xi)\int\limits_\R \left(\frac{-h}{\delta_hf(\xi/2)^2+h^2} - \frac{-h}{\delta_{h}f(-\xi/2)^2+h^2}\right)dh \lesssim \omega'(\xi)\left(\int\limits_0^\xi \frac{\omega(h)}{h}dh + \xi\int\limits_\xi^\infty \frac{\omega(h)}{h^2}dh\right).
\end{equation}

\end{proof}

That leaves us with the second drift term of \eqref{e:driftanddiffusion},
\begin{lemma}\label{l:drifttwo}
Under the assumptions of Lemma \ref{l:omegabound}, for any $M\geq 1$
\begin{equation}
\omega'(\xi)\bigg|\int\limits_{-M\xi}^{M\xi} hK(\xi/2,h) - hK(-\xi/2,h)dh\bigg| \lesssim \omega'(\xi)\left(\int\limits_0^\xi \frac{\omega(h)}{h}dh + \xi\int\limits_\xi^\infty \frac{\omega(h)}{h^2}dh + \ln(M+1)\omega(\xi)\right).
\end{equation}
\end{lemma}

\begin{proof}
To begin, we note
\begin{equation}\label{e:seconddriftterm}
\omega'(\xi)\bigg|\int\limits_{-M\xi}^{M\xi} hK(\xi/2,h) - hK(-\xi/2,h)dh \bigg| \leq \omega'(\xi)\int\limits_0^{M\xi} h\bigg|K(\xi/2,h)-K(\xi/2, -h) - K(-\xi/2,h)+K(-\xi/2,-h)\bigg|dh.
\end{equation}

Recall the definition of $K$, \eqref{e:Kequation},
\begin{equation}
\begin{split}
K(x,h) = \left\{ \begin{array}{cl} \int\limits_h^\infty k(x,s)ds, & h>0 \\ \int\limits_{-\infty}^h -k(x,s) ds, & h<0 \end{array}\right. ,
\\ k(x,s) = \frac{2\left(\delta_sf(x)f_x(x)+s\right)}{\left(\delta_sf(x)^2 + s^2\right)^2}.
\end{split}
\end{equation}
So, to control \eqref{e:seconddriftterm} we first need to bound $|k(x,s)+k(x,-s)|$ for $0\leq s<\xi$, and $|k(\xi/2,s)-k(-\xi/2,s)|$ for $|s|>\xi$.
For the first, using the bounds \eqref{e:fbound1} we see that
\begin{equation}\label{e:kbound1}
\begin{split}
|k(x,s)+k(x,-s)| &= \bigg|\frac{2\left(\delta_sf(x)f_x(x)+s\right)}{\left(\delta_sf(x)^2 + s^2\right)^2} + \frac{2\left(\delta_{-s}f(x)f_x(x)-s\right)}{\left(\delta_{-s}f(x)^2 + s^2\right)^2}\bigg|
\\&\leq \frac{2|\delta_sf(x) + \delta_{-s}f(x)|\cdot |f_x(x)|}{\left(\delta_{-s}f(x)^2 + s^2\right)^2} + 2|\delta_sf(x)f_x(x)+s|\bigg|\frac{\left(\delta_sf(x)^2 + s^2\right)^2 - \left(\delta_{-s}f(x)^2 + s^2\right)^2}{\left(\delta_sf(x)^2 + s^2\right)^2\left(\delta_{-s}f(x)^2 + s^2\right)^2}\bigg|
\\&\lesssim \frac{\omega(s)}{s^3} + s\bigg|\frac{\delta_sf(x)^4-\delta_{-s}f(x)^4 + 2s^2(\delta_sf(x)^2-\delta_{-s}f(x)^2)}{s^8}\bigg|
\\&\lesssim \frac{\omega(s)}{s^3}.
\end{split}
\end{equation}

For the second, using \eqref{e:fbound1}, \eqref{e:fbound2}, and \eqref{e:omegaassumptions} we get that
\begin{equation}\label{e:kbound2}
\begin{split}
|k(\xi/2,s)-k(-\xi/2,s)| &= \bigg|\frac{2\left(\delta_sf(\xi/2)f_x(\xi/2)+s\right)}{\left(\delta_sf(\xi/2)^2 + s^2\right)^2} - \frac{2\left(\delta_{s}f(-\xi/2)f_x(-\xi/2)+s\right)}{\left(\delta_{s}f(-\xi/2)^2 + s^2\right)^2}\bigg|
\\&\leq 2\frac{|\delta_sf(\xi/2)f_x(\xi/2)-\delta_sf(-\xi/2)f_x(-\xi/2)|}{\left(\delta_{s}f(-\xi/2)^2 + s^2\right)^2}
\\&\quad + 2|\delta_sf(\xi/2)f_x(\xi/2)+s| \bigg|\frac{\left(\delta_sf(\xi/2)^2 + s^2\right)^2 - \left(\delta_{s}f(-\xi/2)^2 + s^2\right)^2}{\left(\delta_sf(\xi/2)^2 + s^2\right)^2\left(\delta_{s}f(-\xi/2)^2 + s^2\right)^2}\bigg|
\\&\lesssim  \frac{|\delta_sf(\xi/2)-\delta_sf(-\xi/2)|\cdot |f_x(\xi/2)|}{ s^4} + \frac{|\delta_sf(-\xi/2)|\cdot |f_x(\xi/2)-f_x(-\xi/2)|}{s^4}
\\&\quad + |s| \bigg|\frac{\delta_sf(\xi/2)^4-\delta_sf(-\xi/2)^4 + s^2\left(\delta_sf(\xi/2)^2-\delta_{s}f(-\xi/2)^2\right)}{s^8}\bigg|
\\&\lesssim \frac{\xi\omega(s)}{s^4} + \frac{\omega(\xi)}{s^3}.
\end{split}
\end{equation}

So using \eqref{e:kbound1} and \eqref{e:kbound2}, we can first bound
\begin{equation}\label{e:seconddrifttermsmall}
\begin{split}
\int\limits_0^\xi h\bigg|K(\xi/2,h)-K(\xi/2, -h) - K(-\xi/2,h) &+K(-\xi/2,-h)\bigg|dh
\\ &\lesssim \int\limits_0^\xi h\int\limits_h^\xi \frac{\omega(s)}{s^3} ds dh+ \int\limits_0^\xi h\int\limits_\xi^\infty \frac{\xi\omega(s)}{s^4} + \frac{\omega(\xi)}{s^3}ds dh
\\&\lesssim \int\limits_0^\xi \frac{\omega(s)}{s^3} \int\limits_0^s hdh ds+ \int\limits_{\xi}^\infty \frac{\xi^3\omega(s)}{s^4} +\frac{\xi^2\omega(\xi)}{s^3}ds
\\&\lesssim \int\limits_0^\xi \frac{\omega(s)}{s}ds + \xi\int\limits_\xi^\infty \frac{\omega(s)}{s^2}ds + \omega(\xi).
\end{split}
\end{equation}

For the rest of \eqref{e:seconddriftterm}, we use \eqref{e:kbound2} again to also bound
\begin{equation}
\begin{split}\label{e:seconddrifttermlarge}
\int\limits_{M\xi>|h|>\xi} |h| \ \bigg|K(\xi/2, h) - K(-\xi/2,h)\bigg| dh &\lesssim \int\limits_{\xi}^{M\xi} h\int\limits_h^\infty \frac{\omega(\xi)}{s^3} + \frac{\xi\omega(s)}{s^4} ds
\\&\lesssim \omega(\xi)\int\limits_{\xi}^{M\xi} \frac{1}{h} dh + \xi\int\limits_{\xi}^{M\xi} \frac{\omega(h)}{h^2}dh
\\&\lesssim \ln(M)\omega(\xi) + \xi\int\limits_{\xi}^\infty \frac{\omega(h)}{h^2}dh.
\end{split}
\end{equation}

\end{proof}

\section{Bounds on Diffusive Terms}

Now we move on to proving an upper bound for the diffusive terms of \eqref{e:driftanddiffusion}.  We can rewrite them as
\begin{equation}\label{e:diffusiveterms}
\begin{split}
\int\limits_{-M\xi}^{M\xi} (\delta_hf_x&(\xi/2)-h\omega'(\xi))K(\xi/2,h) - (\delta_hf_x(-\xi/2)-h\omega'(\xi))K(-\xi/2,h)dh
\\&\qquad\qquad\qquad+\int\limits_{|h|>M\xi} \delta_hf_x(\xi/2)K(\xi/2,h) -\delta_hf_x(-\xi/2)K(-\xi/2,h)dh
\\&=\int\limits_{-M\xi}^{M\xi} (\delta_hf_x(\xi/2)-h\omega'(\xi))K(\xi/2,h) - (\delta_hf_x(-\xi/2)-h\omega'(\xi))K(-\xi/2,h)dh
\\ &\quad+\int\limits_{|h|>M\xi} \left[\delta_hf_x(\xi/2) - \delta_h f_x(-\xi/2)\right]K(\xi/2, h)dh + \int\limits_{|h|>M\xi} \delta_hf_x(-\xi/2)\left[K(\xi/2, h)-K(-\xi/2,h)\right] dh.
\end{split}
\end{equation}

We begin by bounding the last term, which is an error term.
\begin{lemma}
Under the assumptions of Lemma \ref{l:omegabound},
\begin{equation}
\bigg|\int\limits_{|h|>M\xi} \delta_hf_x(-\xi/2)\left[K(\xi/2, h)-K(-\xi/2,h)\right]\bigg| dh \lesssim \omega(\xi)\int\limits_{M\xi}^\infty \frac{\omega(h)}{h^2}dh + \omega'(\xi)\xi\int\limits_\xi^\infty \frac{\omega(h)}{h^2}dh.
\end{equation}
\end{lemma}

\begin{proof}
Using the fact that $f_x$ has modulus $\omega$ and the bounds \ref{e:kbound2}, it follows that
\begin{equation}\label{e:diffusionerrorsmall}
\begin{split}
\int\limits_{|h|>M\xi} \delta_hf_x(-\xi/2)\left[K(\xi/2, h)-K(-\xi/2,h)\right] dh &\lesssim\int\limits_{M\xi}^\infty \omega(h)\int\limits_h^\infty \frac{\omega(\xi)}{s^3}+\frac{\xi\omega(s)}{s^4} ds dh
\\&\lesssim  \omega(\xi)\int\limits_{M\xi}^\infty \frac{\omega(h)}{h^2}dh + \int\limits_{M\xi}^\infty \omega(h)\int\limits_h^\infty \frac{\xi\omega(\xi)+\xi\omega'(\xi)(s-\xi)}{s^4}ds dh
\\&\lesssim  \omega(\xi)\int\limits_{M\xi}^\infty \frac{\omega(h)}{h^2}dh + \omega(\xi)\int\limits_{M\xi}^\infty \frac{\xi\omega(h)}{h^3}dh + \omega'(\xi)\xi\int\limits_{M\xi}^\infty \frac{\omega(h)}{h^2}dh
\\&\lesssim \omega(\xi)\int\limits_{M\xi}^\infty \frac{\omega(h)}{h^2}dh.+\omega'(\xi)\xi\int\limits_\xi^\infty \frac{\omega(h)}{h^2}dh.
\end{split}
\end{equation}

\end{proof}


For the other two terms in \eqref{e:diffusiveterms}, we bound them in two stages.

\begin{lemma}
Under the assumptions of Lemma \ref{l:omegabound},
\begin{equation}\label{e:diffusiveestimate1}
\begin{split}
\int\limits_{-M\xi}^{M\xi}&(\delta_hf_x(\xi/2)-h\omega'(\xi))K(\xi/2,h) - (\delta_hf_x(-\xi/2)-h\omega'(\xi))K(-\xi/2,h)dh
\\&+\int\limits_{|h|>M\xi} \left[\delta_hf_x(\xi/2) - \delta_h f_x(-\xi/2)\right]K(\xi/2, h)dh
\\&\qquad\leq \lambda\int\limits_\R \frac{\delta_hf_x(\xi/2) - \delta_h f_x(-\xi/2)}{h^2}dh + 2(\Lambda-\lambda)\int\limits_{\xi}^{M\xi} \frac{(\omega(h-\xi) - \omega(\xi))_+}{h^2}dh
\\&\qquad\qquad\qquad+ \omega'(\xi)\int\limits_{\xi<|h|<M\xi} \bigg|h\left[K(\xi/2,h) - K(-\xi/2,h)\right]\bigg|dh.
\end{split}
\end{equation}
\end{lemma}

\begin{proof}

We can bound the second term of \eqref{e:diffusiveestimate1} rather easily.  Since
\begin{equation}
\delta_hf_x(\xi/2) - \delta_h f_x(-\xi/2) = (f_x(h+\xi/2)-f_x(h-\xi/2)) - \omega(\xi)\leq 0,
\end{equation}
by the uniform ellipticity of $K$,
\begin{equation}
\int\limits_{|h|>M\xi} \left[\delta_hf_x(\xi/2) - \delta_h f_x(-\xi/2)\right]K(\xi/2, h) dh \leq \lambda\int\limits_{|h|>M\xi}\frac{\delta_hf_x(\xi/2) - \delta_h f_x(-\xi/2)}{h^2}dh.
\end{equation}

To bound the first term, we first define
\begin{equation}
G(\xi, h) = (\delta_hf_x(\xi/2)-h\omega'(\xi))K(\xi/2,h) - (\delta_hf_x(-\xi/2)-h\omega'(\xi))K(-\xi/2,h).
\end{equation}
Note that since $\omega$ is concave and touches $f_x$ from above (see \eqref{e:touchfromabove}), it follows that
\begin{equation}
\begin{split}\label{e:touchfromabove2}
\delta_h f_x(\xi/2) - \omega'(\xi)h \leq \delta_h\omega(\xi)-\omega'(\xi)h \leq 0,  \quad h\geq-\xi
\\ \delta_h f_x(-\xi/2) - \omega'(\xi)h \geq -\delta_{-h}\omega(\xi)-h\omega'(\xi)\geq 0, \quad h\leq \xi
\end{split}
\end{equation}

Thus for $|h|\leq\xi$, by the uniform ellipticity of $K$ we have the bound
\begin{equation}
G(\xi,h)\leq \lambda \frac{\delta_hf_x(\xi/2)-\delta_hf_x(-\xi/2)}{h^2}.
\end{equation}

That just leaves us with the case  $\xi\leq |h| \leq M\xi$ to analyze.  Note that we can write $G$ in two distinct ways:
\begin{equation}
\begin{split}
G(\xi,h) &= (\delta_hf_x(\xi/2)-\delta_hf_x(-\xi/2)) K(\xi/2,h) + ( \delta_hf_x(-\xi/2)-h\omega'(\xi))(K(\xi/2,h)-K(-\xi/2,h))
\\&=(\delta_hf_x(\xi/2)-\delta_hf_x(-\xi/2)) K(-\xi/2,h) + ( \delta_hf_x(\xi/2)-h\omega'(\xi))(K(\xi/2,h)-K(-\xi/2,h)).
\end{split}
\end{equation}
By \eqref{e:touchfromabove2}, $\delta_hf_x(\xi/2)-h\omega'(\xi)\leq 0$ for all $h>\xi$.  Thus if $K(\xi/2,h)-K(-\xi/2,h)\geq 0$, then
\begin{equation}
G(\xi,h)\leq \lambda \frac{\delta_hf_x(\xi/2)-\delta_hf_x(-\xi/2)}{h^2}, \quad \quad \mbox{ if }K(\xi/2,h)-K(-\xi/2,h)\geq 0
\end{equation}
On the other hand, since
\begin{equation}
\delta_hf_x(-\xi/2) = \delta_{h-\xi}f(\xi/2) + \omega(\xi) \geq  - \omega(h-\xi) + \omega(\xi)
\end{equation}
for $h\geq\xi$, we see that
\begin{equation}
\begin{split}
G(\xi,h) &\leq \lambda \frac{\delta_hf_x(\xi/2)-\delta_hf_x(-\xi/2)}{h^2} + ( \delta_hf_x(-\xi/2)-h\omega'(\xi))(K(\xi/2,h)-K(-\xi/2,h))
\\& \leq \lambda \frac{\delta_hf_x(\xi/2)-\delta_hf_x(-\xi/2)}{h^2} +(\Lambda-\lambda)\frac{(\omega(h-\xi)-\omega(\xi))_+}{h^2} +h\omega'(\xi)|K(\xi/2,h)-K(-\xi/2,h)|,
\\ & \mbox{ if }K(\xi/2,h)-K(-\xi/2,h)\leq 0.
\end{split}
\end{equation}

Putting these two together, we get that
\begin{equation}
G(\xi,h) \leq \lambda \frac{\delta_hf_x(\xi/2)-\delta_hf_x(-\xi/2)}{h^2} +(\Lambda-\lambda)\frac{(\omega(h-\xi)-\omega(\xi))_+}{h^2} +h\omega'(\xi)|K(\xi/2,h)-K(-\xi/2,h)|.
\end{equation}
for $h\geq\xi$.  A similar argument can be made in the case that $h\leq -\xi$.

Putting this all together,
\begin{equation}
\begin{split}
\int\limits_{-M\xi}^{M\xi} G(\xi,h)dh &+\int\limits_{|h|>M\xi} \left[\delta_hf_x(\xi/2) - \delta_h f_x(-\xi/2)\right]K(\xi/2, h)dh
\\&\leq \lambda\int\limits_\R \frac{\delta_hf_x(\xi/2) - \delta_h f_x(-\xi/2)}{h^2}dh + 2(\Lambda-\lambda)\int\limits_{\xi}^{M\xi} \frac{(\omega(h-\xi) - \omega(\xi))_+}{h^2}dh
\\&\qquad\qquad+ \omega'(\xi)\int\limits_{\xi<|h|<M\xi} \bigg| h\left[K(\xi/2,h) - K(-\xi/2,h)\right]\bigg|dh.
\end{split}
\end{equation}

\end{proof}

It's clear that we can bound $\displaystyle\int\limits_{\xi<|h|<M\xi} \bigg|h\left[K(\xi/2,h) - K(-\xi/2,h)\right]\bigg|dh$ as in \eqref{e:seconddrifttermlarge}.  Thus the only thing remaining to prove \eqref{e:omegabound} is
\begin{lemma}
Under the assumptions of Lemma \ref{l:omegabound},
\begin{equation}
\lambda\int\limits_{\R}\frac{\delta_hf_x(\xi/2) - \delta_hf_x(-\xi/2)}{h^2}dh\leq 2\lambda\int\limits_0^\xi \frac{\delta_h\omega(\xi) + \delta_{-h}\omega(\xi)}{h^2}dh + 2\lambda \int\limits_{\xi}^\infty \frac{\omega(\xi+h)-\omega(h)-\omega(\xi)}{h^2}dh.
\end{equation}
\end{lemma}

\begin{proof}

%

To see this, note that formally we should have
\begin{equation}\label{e:formalintegral}
\begin{split}
\int\limits_{\R}\frac{\delta_hf_x(\xi/2) - \delta_hf_x(-\xi/2)}{h^2}dh = \int\limits_{\R} f_x(y)\left(\frac{1}{(y-\xi/2)^2} - \frac{1}{(y+\xi/2)^2}\right) - \frac{\omega(\xi)}{y^2} dy.
\end{split}
\end{equation}
Thus in order to get an upper bound on \eqref{e:formalintegral}, we should be taking an upper bound on $f_x(y)$ when $y>0$ and a lower bound when $y<0$.  Note by \eqref{e:touchfromabove} that
\begin{equation}\label{e:fybounds}
\begin{split}
f_x(y)\leq f_x(\xi/2) + \omega(y+\xi/2) - \omega(\xi) = f_x(-\xi/2)+\omega(y+\xi/2), \quad y>-\xi/2,
\\ f_x(y)\geq f_x(-\xi/2) -\omega(-y+\xi/2) + \omega(\xi) = f_x(\xi/2)-\omega(-y+\xi/2), \quad y<\xi/2.
\end{split}
\end{equation}
In particular, using the upper bounds bounds on $\delta_hf_x(\pm \xi/2)$ for $h>0$ and the lower bounds for $\delta_hf_x(\pm \xi/2)$ for $h<0$ give the result.  To rigorously justify this though, we will bound $$\displaystyle\int\limits_\epsilon^\infty \frac{\delta_hf_x(\xi/2) - \delta_hf_x(-\xi/2)}{h^2}dh$$ from above.  Taking $\epsilon\to 0$, we'll get 
\begin{equation}
\int\limits_0^\infty \frac{\delta_hf_x(\xi/2) - \delta_hf_x(-\xi/2)}{h^2}dh\leq \int\limits_0^\xi \frac{\delta_h\omega(\xi) + \delta_{-h}\omega(\xi)}{h^2}dh +  \int\limits_{\xi}^\infty \frac{\omega(\xi+h)-\omega(h)-\omega(\xi)}{h^2}dh.
\end{equation}
The bound for $\int\limits_{-\infty}^0$ follows from identical arguments.  

So, fix some $\epsilon<<\xi$.  By splitting the integral into a several pieces and reparameterizing, we get that 
\begin{equation}\label{e:firsthotmess}
\begin{split}
\int\limits_\epsilon^\infty \frac{\delta_hf_x(\xi/2) - \delta_hf_x(-\xi/2)}{h^2}dh = \int\limits_{\epsilon+\xi/2}^\infty \frac{f_x(y)}{(y-\xi/2)^2} dy - \int\limits_{\epsilon - \xi/2}^\infty \frac{f_x(y)}{(y+\xi/2)^2} dy - \int\limits_{\epsilon}^\infty \frac{\omega(\xi)}{y^2}dy
\\= \int\limits_{\epsilon+\xi/2}^\infty f_x(y)\left(\frac{1}{(y-\xi/2)^2} -\frac{1}{(y+\xi/2)^2}\right)dy - \int\limits_{\epsilon}^\infty \frac{\omega(\xi)}{y^2}dy - \int\limits_{\epsilon - \xi/2}^{\epsilon+\xi/2} \frac{f_x(y)}{(y+\xi/2)^2} dy .
\end{split}
\end{equation}
In the first integral of the second line, since $y>\xi/2$ we have that $(y-\xi/2)^{-2}>(y+\xi/2)^{-2}$.  So applying the upper bound in \eqref{e:fybounds} gives an upper bound on the integral,
\begin{equation}\label{e:secondhotmess}
\begin{split}
\int\limits_{\epsilon+\xi/2}^\infty f_x(y)\left(\frac{1}{(y-\xi/2)^2} -\frac{1}{(y+\xi/2)^2}\right)dy \leq \int\limits_{\epsilon+\xi/2}^\infty \left(f_x(\xi/2) + \omega(y+\xi/2) - \omega(\xi)\right)\left(\frac{1}{(y-\xi/2)^2} -\frac{1}{(y+\xi/2)^2}\right)dy 
\\= \int\limits_{\epsilon+\xi/2}^\infty \frac{f_x(\xi/2) + \omega(y+\xi/2) - \omega(\xi)}{(y-\xi/2)^2}dy -\int\limits_{\epsilon + \xi/2}^{\infty}\frac{f_x(\xi/2) + \omega(y+\xi/2) - \omega(\xi)}{(y+\xi/2)^2}dy 
\end{split}
\end{equation}
By reparametrizing back, we get that 
\begin{equation}\label{e:thirdhotmess}
\begin{split}
\int\limits_{\epsilon+3\xi/2}^\infty \frac{f_x(\xi/2) + \omega(y+\xi/2) - \omega(\xi)}{(y-\xi/2)^2}dy -\int\limits_{\epsilon + \xi/2}^{\infty}\frac{f_x(\xi/2) + \omega(y+\xi/2) - \omega(\xi)}{(y+\xi/2)^2}dy -\int\limits_{\epsilon+\xi}^\infty \frac{\omega(\xi)}{y^2}dy 
\\= \int\limits_{\epsilon+\xi}^\infty \frac{\omega(\xi+h)-\omega(h)-\omega(\xi)}{h^2}dh
\end{split}
\end{equation}

Hence combining \eqref{e:firsthotmess},\eqref{e:secondhotmess}, and \eqref{e:thirdhotmess}  gives us 
\begin{equation}
\begin{split}
\int\limits_\epsilon^\infty \frac{\delta_hf_x(\xi/2) - \delta_hf_x(-\xi/2)}{h^2}dh&\leq \int\limits_{\epsilon+\xi}^\infty \frac{\omega(\xi+h)-\omega(h)-\omega(\xi)}{h^2}dh + \int\limits_\epsilon^{\epsilon+\xi} \frac{f_x(\xi/2) + \omega(\xi+h) - \omega(\xi)}{h^2}dh 
\\&\quad- \int\limits_\epsilon^{\epsilon+\xi} \frac{\omega(\xi)}{h^2} dh  - \int\limits_{\epsilon}^{\epsilon+\xi} \frac{f_x(h-\xi/2)}{h^2}dh
\\& =  \int\limits_{\epsilon+\xi}^\infty \frac{\omega(\xi+h)-\omega(h)-\omega(\xi)}{h^2}dh 
\\&\quad+\int\limits_{\epsilon}^{\epsilon+\xi} \frac{\delta_h \omega(\xi) + f_x(\xi/2) - f_x(h-\xi/2)-\omega(\xi)}{h^2}dh.
\end{split}
\end{equation}
Now for $h<\xi$, we have that $f_x(\xi/2) - f_x(h-\xi/2)\leq \omega(\xi-h)$, and thus 
\begin{equation}
\int\limits_{\epsilon}^{\xi} \frac{\delta_h \omega(\xi) + f_x(\xi/2) - f_x(h-\xi/2)-\omega(\xi)}{h^2}dh\leq \int\limits_\epsilon^\xi \frac{\delta_h\omega(\xi) + \delta_{-h}\omega(\xi)}{h^2}dh.  
\end{equation}
Taking the limit as $\epsilon\to 0$, we then get
\begin{equation}
\int\limits_0^\infty \frac{\delta_hf_x(\xi/2) - \delta_hf_x(-\xi/2)}{h^2}dh\leq \int\limits_0^\xi \frac{\delta_h\omega(\xi) + \delta_{-h}\omega(\xi)}{h^2}dh +  \int\limits_{\xi}^\infty \frac{\omega(\xi+h)-\omega(h)-\omega(\xi)}{h^2}dh.
\end{equation}

\end{proof}

\section{Modulus Inequality}
Combining all the estimates from the previous two sections, we get a proof of Lemma \ref{l:omegabound}.  Thus under the assumptions \eqref{e:omegaassumptions}, we have  that
\begin{equation}
\begin{split}
\frac{d}{dt}(f_x(\xi/2)-f_x(-\xi/2)) \leq & A\omega'(\xi)\left(\int\limits_0^\xi \frac{\omega(h)}{h}dh + \xi\int\limits_{\xi}^\infty \frac{\omega(h)}{h^2} dh  + \ln(M+1)\omega(\xi)\right)
\\&+A\omega(\xi)\int\limits_{M\xi}^\infty \frac{\omega(h)}{h^2}dh +2(\Lambda-\lambda)\int\limits_{\xi}^{M\xi} \frac{(\omega(h-\xi) - \omega(\xi))_+}{h^2}dh
\\&+2\lambda\int\limits_0^\xi \frac{\delta_h\omega(\xi) + \delta_{-h}\omega(\xi)}{h^2}dh + 2\lambda\int\limits_{\xi}^\infty \frac{\omega(h+\xi)-\omega(h)-\omega(\xi)}{h^2}dh,
\end{split}
\end{equation}
for any $M\geq 1$, where $A$ is a constant depending only on $||f_0'||_{L^\infty}$.

In \cite{Kiselev}, the authors showed that the modulus
\begin{equation}\label{e:omegadefn}
\left\{\begin{array}{cl} \omega(\xi) = \xi-\xi^{3/2}, & 0\leq \xi \leq \delta \\ \omega'(\xi) = \displaystyle\frac{\gamma}{\xi(4+\log(\xi/\delta))}, & \xi\geq \delta \end{array}\right. ,
\end{equation}
satisfies
\begin{equation}\label{e:Kiselevcalc}
A\omega'(\xi)\left(\int\limits_0^\xi \frac{\omega(h)}{h}dh + \xi\int\limits_{\xi}^\infty \frac{\omega(h)}{h^2} dh \right) +\lambda\int\limits_0^\xi \frac{\delta_h\omega(\xi) + \delta_{-h}\omega(\xi)}{h^2}dh + \lambda\int\limits_{\xi}^\infty \frac{\omega(h+\xi)-\omega(h)-\omega(\xi)}{h^2}dh < 0,
\end{equation}
for all $\xi\in \R$ so long as $\delta, \gamma$ are sufficiently small.

With that in mind, we will show that
\begin{lemma}\label{l:omegainequality}
Under the assumptions of Lemma \ref{l:omegabound} for the modulus $\omega$ defined in \eqref{e:omegadefn},
\begin{equation}
\frac{d}{dt}(f_x(\xi/2)-f_x(-\xi/2)) < -\omega'(\xi)\omega(\xi),
\end{equation}
as long as $\delta, \gamma$ are taken sufficiently small depending on $\beta(f_0'),||f_0'||_{L^\infty}$.
\end{lemma}

\begin{proof}

By the Lemma \ref{l:omegabound} and \eqref{e:Kiselevcalc} which was proven in \cite{Kiselev}, it suffices to show
\begin{equation}
\begin{split}
A\omega'(\xi)\ln(M+1)\omega(\xi) &+A\omega(\xi)\int\limits_{M\xi}^\infty \frac{\omega(h)}{h^2}dh +2(\Lambda-\lambda)\int\limits_{\xi}^{M\xi} \frac{(\omega(h-\xi) - \omega(\xi))_+}{h^2}dh
\\&+\lambda\int\limits_0^\xi \frac{\delta_h\omega(\xi) + \delta_{-h}\omega(\xi)}{h^2}dh + \lambda\int\limits_{\xi}^\infty \frac{\omega(h+\xi)-\omega(h)-\omega(\xi)}{h^2}dh\leq -\omega'(\xi)\omega(\xi)
\end{split}
\end{equation}
for the correct choices of $M$, and $\delta, \gamma$ sufficiently small.

We proceed very similarly to \cite{Kiselev}.
To begin, for $\xi\leq \delta$ we take $M=1$.  Then we just need to show that
\begin{equation}
A\omega'(\xi)\omega(\xi) + A\omega(\xi)\int\limits_{\xi}^\infty \frac{\omega(h)}{h^2}dh + \lambda\int\limits_0^\xi \frac{\delta_h\omega(\xi) + \delta_{-h}\omega(\xi)}{h^2}dh + \lambda\int\limits_{\xi}^\infty \frac{\omega(h+\xi)-\omega(h)-\omega(\xi)}{h^2}dh\leq -\omega'(\xi)\omega(\xi) .
\end{equation}

In this regime, note that we have the bounds
\begin{equation}
\left\{\begin{array}{l}
\int\limits_{\xi}^\delta\frac{\omega(h)}{h^2}dh \leq \log(\delta/\xi),
\\ \int\limits_{\delta}^\infty \frac{\omega(h)}{h^2}dh = \frac{\omega(\delta)}{\delta} + \gamma\int\limits_{\delta}^\infty \frac{1}{h^2(4+\log(h/\delta))} dh \leq 1 + \frac{\gamma}{4\delta} \leq 2 \mbox{ if you take }\gamma<4\delta,
\\ \omega'(\xi)\leq 1,
\\ \omega(\xi)\leq \xi,
\\ \int\limits_0^\xi \frac{\omega(\xi+h)+\omega(\xi-h)-2\omega(\xi)}{h^2} \leq \xi\omega''(\xi) = -\frac{3}{2}\xi\xi^{-1/2}.
\end{array}\right.
\end{equation}
Putting this all together, we get that
\begin{equation}
\begin{split}
(A&+1)\omega'(\xi)\omega(\xi) + A\omega(\xi)\int\limits_{\xi}^\infty \frac{\omega(h)}{h^2}dh+ \lambda \int\limits_0^\xi \frac{\omega(\xi+h) + \omega(\xi-h)-2\omega(\xi)}{h^2}dh
\\&+ \lambda \int\limits_{\xi}^\infty \frac{\omega(\xi+h)-\omega(h)-\omega(\xi)}{h^2}dh
\leq \xi\left( (A+1)(3+\log(\delta/\xi)) - \frac{3}{2}\lambda \xi^{-1/2}\right)<0,
\end{split}
\end{equation}
assuming that $\delta$ is sufficiently small.

Now assume that $\xi\geq \delta$.  Then what we need to show is
\begin{equation}
\begin{split}
A\omega'(\xi)&\ln(M+1)\omega(\xi) + A\omega(\xi)\int\limits_{M\xi}^\infty \frac{\omega(h)}{h^2}dh +2(\Lambda-\lambda)\int\limits_{\xi}^{M\xi} \frac{(\omega(h-\xi) - \omega(\xi))_+}{h^2}dh
\\&+ \lambda\int\limits_0^\xi \frac{\delta_h\omega(\xi) + \delta_{-h}\omega(\xi)}{h^2}dh + \lambda\int\limits_{\xi}^\infty \frac{\omega(h+\xi)-\omega(h)-\omega(\xi)}{h^2}dh\leq -\omega'(\xi)\omega(\xi) .
\end{split}
\end{equation}

We first bound our new error terms.  Using the definition of $\omega$ and integrating by parts, we see that
\begin{equation}
\begin{split}
2(\Lambda-\lambda)\int\limits_{2\xi}^{M\xi} \frac{\omega(h-\xi) - \omega(\xi)}{h^2}dh &\leq 2(\Lambda-\lambda)\int\limits_\xi^\infty \frac{\omega(h)-\omega(\xi)}{h^2} dh
\leq 2(\Lambda-\lambda) \int\limits_\xi^\infty \frac{\gamma}{h^2(4+\log(h/\delta))}dh 
\\&\leq \frac{2(\Lambda-\lambda)\gamma}{\xi}\leq \frac{\lambda}{4}\frac{\omega(\delta)}{\xi}\leq \frac{\lambda}{4}\frac{\omega(\xi)}{\xi},
\end{split}
\end{equation}
assuming $\gamma\leq \frac{\lambda}{8(\Lambda-\lambda)}\omega(\delta)$.

In order to bound our other new error term, we will be taking $M$ sufficiently large and then $\gamma$ sufficiently small depending on $M,\delta$.  Noting that $\omega(\xi)\leq 2||f_0'||_{L^\infty}$, we can bound our other new error term by integrating by parts
\begin{equation}
\begin{split}
A\omega(\xi)\int\limits_{M\xi}^\infty \frac{\omega(h)}{h^2}dh &\leq \frac{2A||f_0'||_{L^\infty}}{M}\frac{\omega(M\xi)}{\xi} +2A||f_0'||_{L^\infty}\int\limits_{M\xi}^\infty \frac{\gamma}{h^2(4+\log(h/\delta))}dh
\\&\leq \frac{2A||f_0'||_{L^\infty}}{M}\frac{\omega(M\xi)}{\xi} +\frac{2A||f_0'||_{L^\infty}}{M}\frac{\gamma}{\xi}
\\&\leq \frac{\lambda}{16}\frac{\omega(M\xi)}{\xi} + \frac{\lambda}{8}\frac{\omega(\xi)}{\xi},
\end{split}
\end{equation}
assuming that $$M \geq \displaystyle\frac{32A||f_0'||_{L^\infty}}{\lambda},$$ and then $\gamma$ is sufficiently small so that $$\displaystyle\frac{2||f_0'||_{L^\infty}A}{M}\gamma\leq \frac{\lambda}{8}\omega(\delta)\leq \frac{\lambda}{8}\omega(\xi).$$  Note that this is where we set a value for $M$, and that $\gamma$ is taken sufficiently small depending on $M$.  Now that the value for $M$ is fixed, we can also control the value $\omega(M\xi)$ by taking $\gamma$ sufficiently small that
\begin{equation}
\begin{split}
\omega(M\xi) &= \omega(\xi) + \int\limits_{\xi}^{M\xi}\frac{\gamma}{h(4+\log(h/\delta))}dh \leq \omega(\xi) + \gamma\ln(M)\leq \omega(\xi)+\omega(\delta)
\\&\leq 2\omega(\xi).
\end{split}
\end{equation}
Hence,
\begin{equation}
A\omega(\xi)\int\limits_{M\xi}^\infty \frac{\omega(h)}{h^2}dh \leq\frac{\lambda}{16}\frac{\omega(M\xi)}{\xi} + \frac{\lambda}{8}\frac{\omega(\xi)}{\xi} \leq \frac{\lambda}{4}\frac{\omega(\xi)}{\xi}.
\end{equation}

Using the same integration by parts tricks, we can also show
\begin{equation}
\lambda\int\limits_{\xi}^\infty \frac{\omega(h+\xi)-\omega(h)-\omega(\xi)}{h^2}dh \leq -\frac{3}{4}\lambda\frac{\omega(\xi)}{\xi}.
\end{equation}
for $\gamma$ sufficiently small.

So combining these together, we get that
\begin{equation}
A\omega(\xi)\int\limits_{M\xi}^\infty \frac{\omega(h)}{h^2}dh +2(\Lambda-\lambda)\int\limits_{2\xi}^{M\xi} \frac{\omega(h-\xi) - \omega(\xi)}{h^2}dh + \lambda\int\limits_{\xi}^\infty \frac{\omega(h+\xi)-\omega(h)-\omega(\xi)}{h^2}dh \leq \frac{-\lambda}{4} \frac{\omega(\xi)}{\xi}.
\end{equation}

Since $\omega'(\xi)\omega(\xi)\leq \displaystyle\frac{\gamma\omega(\xi)}{\xi}$, we finally get that
\begin{equation}
\begin{split}
(A\ln(M+1)+1)\omega'(\xi)\omega(\xi) - \frac{\lambda}{4}\frac{\omega(\xi)}{\xi}
\leq \frac{\omega(\xi)}{\xi}\left((A\ln(M+1)+1)\gamma -\lambda/4\right)<0,
\end{split}
\end{equation}
if $\gamma$ is taken sufficiently small.
\end{proof}

\section{Our choice for the modulus $\rho$}
We've now shown that for the modulus defined in \eqref{e:omegadefn} that if the assumptions \eqref{e:omegaassumptions} hold that
\begin{equation}\label{e:omegafinalequation}
\frac{d}{dt}\left(f_x(t,\xi/2)-f_x(t,-\xi/2)\right)\bigg|_{t=T} < -\omega'(\xi)\omega(\xi).
\end{equation}
We claim that in fact \eqref{e:omegafinalequation} will hold for any rescaling $\omega_r(h) = \omega(rh)$ as well.  To see this, fix some $r>0$, and suppose that $f(t,x)$ satisfies the conditions of Lemma \ref{l:omegabound} for $\omega_r$ at time $T$ and distance $\xi$. Take $\tilde{f} (t,x) = rf(t/r, x/r)$, which is also a solution of \eqref{e:fequation}.  Then $\tilde{f}_x$ is a solution of \eqref{e:gequationK} with $\beta(\tilde{f}_0') = \beta(f_0')$, $||\tilde{f}_0'||_{L^\infty} = ||f_0'||_{L^\infty}$, and satisfying the conditions of Lemma \ref{l:omegabound} for $\omega$ at time $rT$ and distance $r\xi$.  Hence by Lemma \ref{l:omegainequality} 
\begin{equation}
\frac{d}{dt}\left(f_x(t,\xi/2)-f_x(t,-\xi/2)\right)\bigg|_{t=T} = r\frac{d}{dt}\left(\tilde{f}_x(t,r\xi/2)-\tilde{f}_x(t,-r\xi/2)\right)\bigg|_{t=rT} < -r\omega'(r\xi)\omega(r\xi) = -\omega_r'(\xi)\omega_r(\xi).
\end{equation}
So, \eqref{e:omegafinalequation} will hold for any rescaling $\omega_r$.  Also note that for $f_x(T,\xi/2)-f_x(T,-\xi/2)= \omega(\xi)$ to hold, we must necessarily have $\omega(\xi)\leq 2||f_x(T,\cdot)||_{L^\infty}<2||f_0'||_{L^\infty}$.  Thus taking
\begin{equation}
C = \sup\limits_{0<h<\omega^{-1}(2||f_0'||_{L^\infty})} \frac{h}{\omega(h)} = \frac{\omega^{-1}(2||f_0'||_{L^\infty})}{2||f_0'||_{L^\infty}},
\end{equation}
we see that
\begin{equation}
\omega(h)\geq \frac{h}{C}.
\end{equation}
for all relevant $h$.
Define
\begin{equation}\label{e:rhodefn}
\rho(h):= \omega(Ch),
\end{equation}
so that
\begin{equation}
\rho(h)\geq h,
\end{equation}
for all $h\in [0,\rho^{-1}(2||f_0'||_{L^\infty})]$.

Now, suppose that at time $T$, $f$ satisfies the assumptions \eqref{e:omegaassumptions} for $\rho(\cdot/T)$.  Then since $\rho(\cdot/T)$ is a rescaling of $\omega$, we have that
\begin{equation}\label{e:rhofinalequation}
\frac{d}{dt}\left(f_x(T,\xi/2)-f_x(T,-\xi/2)\right) < -\frac{d}{dh}\rho(h/T)\bigg|_{h=\xi}\rho(\xi/T) = \frac{-1}{T}\rho'(\xi/T)\rho(\xi/T)\leq \frac{-\xi}{T^2}\rho'(\xi/T) = \frac{d}{dt}\rho(\xi/t)\bigg|_{t=T}.
\end{equation}
Thus we've constructed a modulus $\rho$ which satisfies \eqref{e:finalinequality}, completing the proof of the generation of a Lipschitz modulus of continuity \eqref{e:mainresult} in our main theorem.

\section{Regularity in Time}

With the construction of the modulus $\rho$, we get universal Lipschitz bounds in space for $f_x(t,\cdot)$.  By the structure of \eqref{e:fequation}, we also get regularity in space for $f_t$.
\begin{proposition}\label{p:regularityspace}
Let $f:(0,T)\times \R \to \R$ be a classical solution to \eqref{e:fequation} with $||f(t,\cdot)||_{W^{1,\infty}}$ bounded and $||f_{xx}(t,\cdot)||_{L^\infty}\lesssim 1/t.$.  Then $f_t(t,\cdot)$ is Log-Lipschitz in space with
\begin{equation}
|f_t(t,\cdot) | \lesssim \max\{-\log(t), 1\}, \qquad |f_t(t,x)-f_t(t,y)| \lesssim -\log(|x-y|)|x-y|\left(1+\frac{1}{t}\right)\quad 0<|x-y|<1/2.
\end{equation}
\end{proposition}

\begin{proof}
For $t<1$, we have that
\begin{equation}
\begin{split}
|f_t(t,x)| &= \bigg|\int\limits_\R \frac{\delta_hf(t,x) - hf_x(t,x)}{\delta_hf(t,x)^2+h^2} dh\bigg|  \leq \bigg|\int\limits_0^\infty \frac{\delta_hf(t,x) +\delta_{-h}f(t,x)}{\delta_{-h}f(t,x)^2 +h^2}dh\bigg|
\\&\qquad+ \bigg|\int\limits_0^\infty \frac{(\delta_hf(t,x)-hf_x(t,x))(\delta_hf(t,x)^2-\delta_{-h}f(t,x)^2)}{(\delta_hf(t,x)^2+h^2)(\delta_{-h}f(t,x)^2 +h^2)} dh\bigg|
\\& \lesssim \int\limits_0^t \frac{1}{t}dh  + \int\limits_t^1\frac{1}{h}dh   +  \int\limits_1^\infty\frac{1}{h^2}+\frac{1}{h^3}dh \lesssim -\log(t) + 1.
\end{split}
\end{equation}
For $t>1$, you can similarly show $|f_t(t,x)|\lesssim 1$, proving the first bound.

For regularity in space, we see that
\begin{equation}
\begin{split}
f_t(t,x)&-f_t(t,y)  = \int\limits_\R \frac{\delta_hf(t,x) - h f_x(t,x)}{\delta_hf(t,x)^2 + h^2} - \frac{\delta_hf(t,y) - h f_x(t,y)}{\delta_hf(t,y)^2 + h^2} dh
\\&= \int\limits_\R \frac{\delta_h f(t,x) - hf_x(t,x) - (\delta_hf(t,y)-hf_x(t,y))}{\delta_h f(t,y)^2+h^2}+ \frac{(\delta_hf(t,x) - h f_x(t,x))(\delta_hf(t,x)^2-\delta_hf(t,y)^2)}{( \delta_hf(t,x)^2 + h^2)(\delta_hf(t,y)^2 + h^2)} dh
\\&\leq \bigg|\int\limits_{|h|<|x-y|}\bigg| + \bigg|\int\limits_{|x-y|<|h|<1} \bigg|+\bigg|\int\limits_{|h|>1}\bigg|
\end{split}
\end{equation}

For $|h|<|x-y|$, we can bound similarly to before to get that
\begin{equation}
\bigg|\int\limits_{|h|<|x-y|}\bigg|\lesssim \int\limits_0^{|x-y|} \frac{1}{t}dh = \frac{|x-y|}{t}.
\end{equation}
For midsize $|x-y| < |h| < 1$, we have that
\begin{equation}
\begin{split}
\bigg|\delta_h f(t,x) - hf_x(t,x) - (\delta_hf(t,y)-hf_x(t,y))\bigg| = \bigg|\int\limits_0^h \delta_sf_x(t,x)- \delta_sf_x(t,y)ds \bigg|\lesssim \frac{|x-y|h}{t},
\\ \bigg| \delta_hf(t,x) - \delta_hf(t,y)\bigg| = \bigg| \int\limits_0^h f_x(t,x+s) - f_x(t,y+s)ds\bigg|\lesssim \frac{|x-y|h}{t}.
\end{split}
\end{equation}
Thus
\begin{equation}
 \bigg|\int\limits_{|x-y|<|h|<1} \bigg|\lesssim \frac{|x-y|}{t}\int\limits_{|x-y|}^1 \frac{1}{h}dh= \frac{-\ln(|x-y|)|x-y|}{t}.
\end{equation}
Finally, we use $L^\infty$ bounds on $f$ to get that
\begin{equation}
\begin{split}
\bigg|\int\limits_{|h|>1}\bigg| &\leq \bigg|\int\limits_{|h|>1}  \frac{\delta_h f(t,x) - \delta_hf(t,y)}{\delta_h f(t,y)^2+h^2} + \frac{(\delta_hf(t,x) - h f_x(t,x))(\delta_hf(t,x)^2-\delta_hf(t,y)^2)}{( \delta_hf(t,x)^2 + h^2)(\delta_hf(t,y)^2 + h^2)} dh\bigg|
\\&\qquad +|f_x(t,x)-f_x(t,y)| \ \bigg|\int\limits_{|h|>1} \frac{-h}{\delta_hf(t,y)^2 +h^2} dh \bigg|
\\& \lesssim |x-y|\int\limits_1^\infty \frac{1}{h^2} + \frac{1}{h^3} dh + \frac{|x-y|}{t}\int\limits_1^\infty \frac{1}{h^3}dh \lesssim \left(1+\frac{1}{t}\right)|x-y|.
\end{split}
\end{equation}
Putting this all together, we thus have that
\begin{equation}
|f_t(t,x)-f_t(t,y)|\lesssim -\ln(|x-y|)|x-y|\left(1+\frac{1}{t}\right).
\end{equation}
\end{proof}

Recall that in section 2, we assumed that our initial data $f_0\in C^\infty_c(\R)$ so that by the local existence results of \cite{MaxPrinciple}, there is a unique solution $f\in C^1((0,T_+); H^k)$ for $k$ arbitrarily large and some $T_+>0$.  We were then able to prove the existence of the modulus $\rho$ as in Theorem \ref{t:main} depending only on $\beta(f_0'),||f_0'||_{L^\infty}$, and hence with the solution $f$ existing for all time by the main theorem of \cite{ConstantinMain}.  For an arbitrary $f_0 \in W^{1,\infty}(\R)$ with $\beta(f_0')<1$, the same result holds true by compactness.  Let $\eta\in C^\infty_c(\R)$ be a smooth mollifier, and $\phi\in C^\infty_c(\R)$ be a smooth cutoff function.  For $f_0\in W^{1,\infty}(\R)$ with $\beta(f_0')<1$, take $f_0^{(\epsilon)}(x):= (f_0*\eta_\epsilon)(x)\phi(\epsilon x)$.  Then $f_0^{(\epsilon)}\to f_0$ in $W^{1,\infty}_{loc}$, with $\beta(f_0^{(\epsilon)\prime}), ||f_0^{(\epsilon)}||_{W^{1,\infty}(\R)} \to \beta(f_0'),||f_0||_{W^{1,\infty}(\R)}$ respectively as $\epsilon \to 0$.  Thus for $\epsilon $ sufficiently small, $\beta(f_0^{(\epsilon)\prime } )<1$ and the results of the previous section hold for the solution to the mollified problem $f^{(\epsilon)}$.  The $L^\infty$ bound on $f_t^{(\epsilon)}$ proven above along with the maximum principle for $f_x^{(\epsilon)}$ is enough to ensure that there a subsequence $f^{(\epsilon_k)}$ converging in $C_{loc}([0,\infty)\times \R)$ to a Lipschitz (weak) solution $f$ to the original problem.  In order to get a classical $C^1$ solution, we need regularity estimates for $f_x^{(\epsilon)},f_t^{(\epsilon)}$ in both time and space.  The modulus $\rho$ and Proposition \ref{p:regularityspace} give the regularity in space that we need for $f_x,f_t$.  All that leaves is to prove regularity in time.


\begin{proposition}
Let $f$ be a sufficiently smooth solution to \eqref{e:fequation} with $\beta(f_0')<1$.  Then $f_x,f_t \in C^{\alpha}_{loc}((0,\infty)\times \R)$ with
\begin{equation}
||f_x||_{C^{\alpha}(Q_{t/4}(t,x))}, ||f_t||_{C^{\alpha}(Q_{t/4}(t,x))}\leq C(\beta(f_0'),||f||_{L_t^\infty((t/2,3t/2); W_x^{2,\infty}(\R))}) \max\{t^{-\alpha},1\},
\end{equation}
where $Q_r(s,y) = (s-r,s]\times B_r(y)$, and $\alpha>0$ depends only on $\beta(f_0'),||f_0'||_{L^\infty}$.
\end{proposition}

\begin{proof}

We have that $f_x$ solves
\begin{equation}
(f_x)_t (t,x) =  f_{xx}(t,x) \int\limits_{\R}\frac{-h}{\delta_hf(t,x)^2+h^2} dh + \int\limits_\R \delta_hf_x(t,x) K(t,x,h) dh,
\end{equation}
where $\displaystyle\frac{\lambda}{h^2} \leq K(t,x,h) \leq \frac{\Lambda}{h^2} $ is uniformly elliptic with ellipticity constants $\lambda, \Lambda$ depending on $\beta(f_0'),||f_0'||_{L^\infty}$.
Rewriting this, we have that $f_x$ satisfies
\begin{equation}\label{e:gequationsymmetric}
\begin{split}
(f_x)_t - \int\limits_{\R} \delta_h f_x(t,x) \left(\frac{K(t,x,h)+K(t,x,-h)}{2}\right) dh &=  f_{xx}(t,x) \int\limits_{\R}\frac{-h}{\delta_hf(t,x)^2+h^2} dh
\\&\qquad+ \int\limits_{\R} \delta_h f_x(t,x) \left(\frac{K(t,x,h)-K(t,x,-h)}{2}\right) dh.
\end{split}
\end{equation}

Let $F(t,x)$ denote the righthand side of \eqref{e:gequationsymmetric}.  Then $F(t,x)$ is locally bounded with $|F(t,x)|$ controlled by $||f(t,\cdot)||_{W^{2,\infty}}$.  Then since $(K(t,x,h)+K(t,x,-h))/2$ is a symmetric uniformly elliptic kernel, it follows that we have local $C^{\alpha}$ bounds for $\alpha \leq\alpha_0$ for some $\alpha_0$ depending on ellipticity constants (see \cite{Silvestre}).

So, all we have to do is give bounds on $F(t,x)$ depending only on $||f(t,\cdot)||_{W^{2,\infty}}$.  Similar to proof of Lemma \ref{l:driftone},
\begin{equation}
\int\limits_\R \frac{-h}{\delta_hf(t,x)^2 +h^2} dh = \int\limits_0^\infty h \frac{\delta_hf(t,x)^2 - \delta_{-h}f(t,x)^2}{(\delta_hf(t,x)^2 +h^2)(\delta_{-h}f(t,x)^2 +h^2)} dh\lesssim \int\limits_0^1 1 dh + \int\limits_1^\infty \frac{1}{h^3} dh \lesssim 1.
\end{equation}
Also similar to the proof of Lemma \ref{l:drifttwo} (specifically \eqref{e:kbound1}), we have that
\begin{equation}
|K(t,x,h) - K(t,x,-h)| \lesssim \min\{ \frac{1}{h}, \frac{1}{h^3}\},
\end{equation}
so
\begin{equation}
\bigg|\int\limits_\R \delta_h f_x(t,x) \left(\frac{K(t,x,h)-K(t,x,-h)}{2}\right) dh \bigg| \lesssim \int\limits_0^1 1 dh + \int\limits_1^\infty \frac{1}{h^3} dh \lesssim 1.
\end{equation}

Thus since we've bounded the right hand side of \eqref{e:gequationsymmetric} depending only on $||f(t,\cdot)||_{W^{2,\infty}}$, we have our local $C^\alpha$ bounds for $f_x$ for all $\alpha$ sufficiently small.  A $C^\alpha$ bound that is uniform in $x$ for $f_x$ then gives a log $C^\alpha$ estimate for $f_t$, similar to the proof for regularity in space in Proposition \ref{p:regularityspace}.  Thus we have $C^\alpha$ estimates for both $f_x,f_t$.

\end{proof}

\appendix

\section{Uniqueness}

We now prove that if our initial data $f_0\in C^{1,\epsilon}(\R)$ with $\beta(f_0') <1$, then the solution $f$ given by Theorem \ref{t:main} is unique with $f\in L^\infty([0,\infty) ; C^{1,\epsilon})$.  As mentioned before, this essentially follows from the uniqueness theorem given in \cite{ConstantinMain}, which under our assumptions simplifies to
\begin{theorem}(Constantin et al) \label{t:ConstantinUnique}
Let $f\in L^\infty ([0,T]; W^{1,\infty})$ be a classical, $C^1$ solution to \eqref{e:fequation} with initial data $f(0,x)=f_0(x)$.  Assume that $\lim\limits_{x\to \infty}f(t,x) = 0$, and that there is some modulus of continuity $\tilde{\rho}$ such that
\begin{equation}
f_x(t,x)-f_x(t,y)\leq \tilde{\rho}(|x-y|), \quad \forall 0\leq t\leq T, \ x\not=y\in \R.
\end{equation}
Then the solution $f$ is unique.
\end{theorem}

The authors of \cite{ConstantinMain} note that the uniform continuity assumption should be the only real assumption; the decay is assumed for convenience in their proof. 
So, we start by proving that if $f_0\in C^{1,\epsilon}(\R)$, then the solution $f\in L^\infty ([0,\infty); C^{1,\epsilon})$.  To begin, suppose that $f_0\in C^{1,1}(\R)$. Then necessarily $f_0'$ has modulus $\rho(\cdot/\delta)$ for some $\delta>0$ sufficiently small.  The same proof for the instantaneous generation of the modulus $\rho$ will give that $f_x(t,\cdot)$ has modulus $\rho(\cdot/t+\delta)$.  Hence $f_x(t,\cdot)$ has modulus $\rho(\cdot/\delta)$ for all $t\geq 0$.

If $f_0\in C^{1,\epsilon}(\R)$, we can make the same essential argument by changing the definition of $\rho$ , $\omega$.  You can repeat the arguments of section 7 and 8
for the modulus
\begin{equation}
\left\{\begin{array}{cl} \omega^{(\epsilon)}(\xi) = \xi^{\epsilon}, & 0\leq \xi \leq \delta \\ \omega^{(\epsilon) \ \prime}(\xi) = \displaystyle\frac{\gamma}{\xi(4+\log(\xi/\delta))}, & \xi\geq \delta \end{array}\right. .
\end{equation}
All the error terms for $\xi\leq \delta$ are of order $\xi^{2\epsilon-1}$, while the diffusion term is of the order $\xi^{\epsilon -1}$, so there are no problems as long as $\delta$ is sufficiently small.  The argument for $\xi\geq \delta$ is identical to the original. Taking $\rho^{(\epsilon)}$ to be some suitable rescaling of $\omega^{(\epsilon)}$, we then have that if $f_0'$ has modulus $\rho^{(\epsilon)}(\cdot/\delta)$, then $f_x(t,\cdot)$ will have modulus $\rho^{(\epsilon)}(\cdot/t+\delta)$.

Thus if $f_0\in C^{1,\epsilon}(\R)$, then the solution $f$ given by Theorem \ref{t:main} will satisfy the main uniform continuity assumption of Theorem \ref{t:ConstantinUnique}.  Our solution $f$ will not decay as $x\to \infty$, but that assumption isn't truly necessary.

Let $f_1,f_2$ be two uniformly continuous, classical solutions to \eqref{e:fequation} with the same initial data, and let $M(t) = ||f_1(t,\cdot)-f_2(t,\cdot)||_{L^\infty}$.  With the decay assumption, the authors of \cite{ConstantinMain} are able to assume that for almost every $t$, there is a point $x(t)\in \R$ such that
\begin{equation}
M(t) = |f_1(t,x(t)) - f_2(t,x(t))|, \quad \frac{d}{dt}M(t) = \left(\frac{d}{dt}|f_1 - f_2|\right)(t,x(t)).
\end{equation}
They then bound $\frac{d}{dt}|f_1(t,x(t)) - f_2(t,x(t))|$ using equation $\eqref{e:fequation}$, $\tilde{\rho}$, and $W^{1,\infty}$ bounds.

Without the decay assumption, you instead use that
\begin{equation}
\frac{d}{dt}M(t) \leq \sup\{\frac{d}{dt}|f_1(t,x) - f_2(t,x)| : |f_1(t,x)-f_2(t,x)| \geq M(t)-\delta\},
\end{equation}
where $\delta>0$ is arbitrary.  When you go to bound $\frac{d}{dt}|f_1(t,x) - f_2(t,x)|$, you then get new error terms which can be bounded by
\begin{equation}
C(\tilde{\rho},\max\limits_i ||f_i(t,\cdot)||_{W^{1,\infty}},  M(t)) \left(\delta + |f_{1,x}(t,x)-f_{2,x}(t,x)|\right).
\end{equation}
Since $f_{i,x}(t,x)$ is bounded and has modulus $\tilde{\rho}$, it then follows that
\begin{equation}
|f_{1,x}(t,x)-f_{2,x}(t,x)| = o_\delta(1).
\end{equation}
Thus by taking $\delta$ sufficiently small depending on $\tilde{\rho},\max\limits_i ||f_i(t,\cdot)||_{W^{1,\infty}},  M(t)$, we can guarantee that the new error terms $\lesssim M(t)$.  Then the original proof of \cite{ConstantinMain} goes through.

\section*{Acknowledgements}
I would like to thank my advisor Luis Silvestre for suggesting the problem, pointing me towards good resources, and just giving good advice in general.

\bibliographystyle{abbrv}
\bibliography{MuskatProblem}

\end{document}